\documentclass{birkjour}

\usepackage{amsfonts}
\usepackage{amssymb}
\usepackage{amsmath}
\usepackage{xspace} 

\usepackage{mathptmx}       
\usepackage{helvet}         
\usepackage{courier}        
\usepackage{type1cm}        
\usepackage{graphicx}        
\usepackage{multicol}        
\usepackage[bottom]{footmisc}
\usepackage[noadjust]{cite}
\usepackage{url}

 \newtheorem{thm}{Theorem}[section]
 \newtheorem{lem}[thm]{Lemma}
 \newtheorem{defn}[thm]{Definition}
 \newtheorem{rem}[thm]{Remark}
 \newtheorem{prop}[thm]{Proposition}
 \newtheorem{cor}[thm]{Corollary}
\newcommand{\be}{\begin{equation}}

\renewcommand{\O}{\mathrm{O}}

\newcommand{\cl}[1]{\ensuremath{Cl(#1)}} 

\newcommand{\vect}[1]{\ensuremath{\mbox{\textbf{\textit{#1}}}}}

  \newcommand{\field}[1]{\mathbb{#1}}
  \newcommand{\C}{\field{C}}
  \newcommand{\R}{\field{R}}  
  \newcommand{\Z}{\field{Z}} 

  \newcommand{\HQ}{\field{H}}
  \renewcommand{\H}{\ensuremath{\mathbb{H}}\xspace}
  \newcommand{\svect}[1]{\ensuremath{\mbox{\textbf{\textit{\small #1}}}}}
  \newcommand{\bvect}[1]{\vect{#1}}
  \newcommand{\bomega}{\boldsymbol{\omega}}
  \newcommand{\A}{\mathcal{A}}
\newcommand{\M}{\mathcal{M}}
\newcommand{\Aut}{\mathrm{Aut}}
\newcommand{\Inn}{\mathrm{Inn}}
\newcommand{\G}{\mathrm{G}}

\newcommand{\Spec}{\mathrm{Spec}}
\newcommand{\unitm}{\mathbf{1}} 
\newcommand{\sandwich}{(\,)} 

\DeclareMathOperator{\mymod}{\mathrm{mod}}

\begin{document}

\title[Two-sided Clifford FT with two square roots of $-1$ in $\cl{p,q}$]{Two-sided Clifford Fourier transform with two square roots of $-1$ in $\cl{p,q}$}
\author{Eckhard Hitzer}
\address{%
College of Liberal Arts, Department of Material Science,\\ 
International Christian University,\\
Tokyo 181-8585,\\ 
Japan}
\email{hitzer@icu.ac.jp}
%
%
\subjclass{Primary 42B10; \\Secondary 15A66}
\keywords{Clifford Fourier transform, Clifford algebra, signal processing, square roots of $-1$}
\date{December 1, 2012}
\dedicatory{In memory of our dear friend Rev. Olaug Hansen.}

\begin{abstract}
We generalize quaternion and Clifford Fourier transforms to general two-sided Clifford Fourier transforms (CFT), and study their properties (from linearity to convolution). Two general \textit{multivector square roots} $\in \cl{p,q}$ \textit{of} $-1$ are used to split multivector signals, and to construct the left and right CFT kernel factors. 
\end{abstract}

\maketitle

\section{Introduction}

Quaternion, Clifford and geometric algebra Fourier transforms (QFT, CFT, GAFT) \cite{BSH:genGFT,HM:CFToMVF,EH:QFTgen,EH:OPS-QFT,HNK:AppCliffGA,BDS:FTinCA,HB:CliffAFTQM,HS:Orth2dPSplit} have proven \textit{very useful} tools for applications in non-marginal color image processing, image diffusion,  electromagnetism, multi-channel processing, vector field processing, shape representation, linear scale invariant filtering, fast vector pattern matching, phase correlation, quantum mechanics, analysis of non-stationary improper complex signals, flow analysis, partial differential systems, disparity estimation, texture segmentation, as spectral representations for Clifford wavelet analysis, etc. 

All these Fourier transforms essentially analyze scalar, vector and multivector signals in terms of sine and cosine waves with multivector coefficients. For this purpose the imaginary unit $i\in \C$ in $e^{i \phi} = \cos \phi + i \sin \phi$ can be replaced by any \textit{square root of $-1$ in a Clifford algebra $\cl{p,q}$}. The replacement by pure quaternions and blades with negative square \cite{EH:QFTgen,BSH:genGFT} has already yielded a wide variety of results with a clear geometric interpretation. It is well-known that there are elements other than blades, squaring to $-1$. Motivated by their special relevance for new types of CFTs, they have recently been studied thoroughly \cite{SJS:Biqroots,HA:GeoRoots-1,HHA:ICCA9}.  

We therefore tap into these new results on square roots of $-1$ in Clifford algebras and fully general construct CFTs, with two general square roots of $-1$ in $\cl{p,q}$. Our new CFTs form therefore a more general class of CFTs, subsuming and generalizing previous results. A further benefit is, that these new CFTs become \textit{fully steerable} within the continuous Clifford algebra submanifolds of square roots of $-1$. We thus obtain a comprehensive \textit{new mathematical framework} for the investigation and application of Clifford Fourier transforms together with \textit{new properties} (full steerability). Regarding the question of the \textit{most suitable} CFT for a certain application\footnote{Note in this context the spinor representation of images by Batard and Berthier \cite{BB:TIMBirk}, preceeded by \cite{BB:AGACSE2008}. The authors apply a CFT in $\cl{3,0}$ to the spinor representation, which uses in the exponential kernel an adapted choice of bivector, that belongs to the orthonormal frame of the tangent bundle of an oriented two-dimensional Riemannian manifold (the two-dimensional grey scale image, with the intensity value in the third dimension), isometrically immersed in $\R^3$.}, we are only just beginning to leave the terra cognita of familiar transforms to map out the vast array of possible CFTs in $\cl{p,q}$.  

This paper is organized as follows. We first review in Section \ref{sc:CliffAlg} key notions of Clifford algebra, \textit{multivector signal functions}, and the recent results on \textit{square roots of $-1$} in Clifford algebras. In Section \ref{sc:pmsplit} we show how the $\pm$ split or orthogonal 2D planes split of quaternions can be generalized to \textit{split multivector signal functions} with respect to a general pair of square roots of $-1$ in Clifford algebra. Next, in Section \ref{sc:Gen2sCFT} we define the central notion of \textit{general two-sided Clifford Fourier transforms} with respect to any two square roots of $-1$ in Clifford algebra. Finally, in Section \ref{sc:CFTprop} we investigate the \textit{properties} of these new CFTs: linearity, shift, modulation, dilation, moments, inversion, derivatives, Plancherel and Parseval formulas, as well as a convolution theorem.

\section{Clifford's geometric algebra}
\label{sc:CliffAlg}

\begin{defn}[Clifford's geometric algebra \cite{FM:ICNAAM2007,PL:CAandSpin}] \label{df:CliffAlg}
Let  $\{  e_1, e_2, \ldots , e_p, e_{p+1}, \ldots$, $e_n  \}$, with $n=p+q$, $e_k^2=\varepsilon_k$, $\varepsilon_k = +1$ for $k=1, \ldots , p$, $\varepsilon_k = -1$ for $k=p+1, \ldots , n$, be an \textit{orthonormal base} of the inner product vector space $\R^{p,q}$ with a geometric product according to the multiplication rules 
\be
  e_k e_l + e_l e_k = 2 \varepsilon_k \delta_{k,l}, 
  \qquad k,l = 1, \ldots n,
\label{eq:mrules}
\end{equation}
where $\delta_{k,l}$ is the Kronecker symbol with $\delta_{k,l}= 1$ for $k=l$, and $\delta_{k,l}= 0$ for $k\neq l$. This non-commutative product and the additional axiom of \textit{associativity} generate the $2^n$-dimensional Clifford geometric algebra $Cl(p,q) = Cl(\R^{p,q}) = Cl_{p,q} = \mathcal{G}_{p,q} = \R_{p,q}$ over $\R$. The set $\{ e_A: A\subseteq \{1, \ldots ,n\}\}$ with $e_A = e_{h_1}e_{h_2}\ldots e_{h_k}$, $1 \leq h_1< \ldots < h_k \leq n$, $e_{\emptyset}=1$, forms a graded (blade) basis of $Cl(p,q)$. The grades $k$ range from $0$ for scalars, $1$ for vectors, $2$ for bivectors, $s$ for $s$-vectors, up to $n$ for pseudoscalars. 
The vector space $\R^{p,q}$ is included in $Cl(p,q)$ as the subset of 1-vectors. The general elements of $Cl(p,q)$ are real linear combinations of basis blades $e_A$, called Clifford numbers, multivectors or hypercomplex numbers.
\end{defn}

In general $\langle A \rangle_{k}$ denotes the grade $k$ part of $A\in Cl(p,q)$. The parts of grade $0$ and $k+s$, respectively, of the geometric product of a $k$-vector $A_k\in Cl(p,q)$ with an $s$-vector $B_s\in Cl(p,q)$ 
\begin{gather}
  A_k \ast B_s := \langle A_k B_s \rangle_{0}, 
  \qquad
  A_k \wedge B_s := \langle A_k B_s \rangle_{k+s},
  \label{eq:gaprods}
\end{gather}
are called \textit{scalar product} and \textit{outer product}, respectively.

For Euclidean vector spaces $(n=p)$ we use $\R^{n}=\R^{n,0}$ and $Cl(n) = Cl(n,0)$. Every $k$-vector $B$ that can be written as the outer product $B = \vect{b}_1 \wedge \vect{b}_2 \wedge \ldots \wedge \vect{b}_k$ of $k$ vectors $\vect{b}_1, \vect{b}_2, \ldots, \vect{b}_k \in \R^{p,q}$ is called a \textit{simple} $k$-vector or \textit{blade}. 

Multivectors
  $M \in \cl{p,q}$ have $k$-vector parts ($0\leq k \leq n$):
  {scalar} part
  $Sc(M) = \langle M \rangle = \langle M \rangle_0 = M_0 \in \R$, 
  {vector} part
  $\langle M \rangle_1 \in \R^{p,q}$, 
  {bi-vector} part
  $\langle M \rangle_2$,  \ldots, 
  and
  {pseudoscalar} part $\langle M \rangle_n\in\bigwedge^n\R^{p,q}$
\begin{equation}\label{eq:MVgrades}
    M  =  \sum_{A} M_{A} \vect{e}_{A}
       =  \langle M \rangle + \langle M \rangle_1 + \langle M \rangle_2 + \ldots +\langle M \rangle_n \, .
\end{equation}

The \textit{principal reverse} of $M \in \cl{p,q}$ defined as
\begin{equation}\label{eq:MVrev}
  \widetilde{M}=\; \sum_{k=0}^{n}(-1)^{\frac{k(k-1)}{2}}\langle \overline{M} \rangle_k,
\end{equation}
often replaces {complex conjugation and quaternion conjugation}. Taking the \textit{reverse} is equivalent to reversing the order of products of basis vectors in the basis blades $e_A$. The operation $\overline{M}$ means to change in the basis decomposition of $M$ the sign of every vector of negative square $\overline{e_A} = \varepsilon_{h_1}e_{h_1}\varepsilon_{h_2}e_{h_2}\ldots \varepsilon_{h_k}e_{h_k}$, $1 \leq h_1< \ldots < h_k \leq n$. Reversion, $\overline{M}$, and principal reversion are all involutions. 

For $M,N \in \cl{p,q}$ we get $M\ast \widetilde{N}=\sum_{A} M_A N_A.$
  Two multivectors $M,N \in \cl{p,q}$ are \textit{orthogonal} if and only if $M\ast \widetilde{N} = 0$.
  The {modulus} $|M|$ of a multivector $M \in \cl{p,q}$ is defined as 
  \be
     |M|^2 = {M\ast\widetilde{M}}= {\sum_{A} M_A^2}.
  \end{equation}

\subsection{Multivector signal functions}

A multivector valued function
  $f: \R^{p,q} \rightarrow \cl{p,q}$, has $2^n$ blade components
  $(f_A: \R^{p,q} \rightarrow \R)$
  \begin{equation}\label{eq:MVfunc}
    f(\mbox{\textbf{\textit{x}}})  =  \sum_{A} f_{A}(\vect{x}) {\vect{e}}_{A}.
  \end{equation}
We define the \textit{inner product} of two
 functions  $f, g : \R^{p,q} \rightarrow \cl{p,q}$ by
\begin{align}
  \label{eq:mc2}
  (f,g) 
  = \int_{\R^{p,q}}f(\vect{x})
    \widetilde{g(\vect{x})}\;d^n\vect{x}
  = \sum_{A,B}\vect{e}_A \widetilde{\vect{e}_B}
    \int_{\R^{p,q}}f_A (\vect{x})
    g_B (\vect{x})\;d^n\vect{x},
\end{align}
with the \textit{symmetric scalar part}
\begin{align}
  \label{eq:symsc}
  \langle f,g\rangle 
  = \int_{\R^{p,q}}f(\vect{x})\ast\widetilde{g(\vect{x})}\;d^n\vect{x}
  = \sum_{A}
    \int_{\R^{p,q}}f_A (\vect{x})g_A (\vect{x})\;d^n\vect{x},
\end{align}
and the $L^2(\mathbb{R}^{p,q};\cl{p,q})$-\textit{norm} 
\begin{align}\label{eq:0mc2}
   \|f\|^2 
   = \left\langle ( f,f ) \right\rangle
   &= \int_{\R^{p,q}} |f(\vect{x})|^2 d^n\vect{x}
   = \sum_{A} \int_{\R^{p,q}} f_A^2(\vect{x})\;d^n\vect{x},
   \\
   L^2(\R^{p,q};\cl{p,q})
   &= \{f: \R^{p,q} \rightarrow \cl{p,q} \mid \|f\| < \infty \}. 
\end{align}

\subsection{Square roots of $-1$ in Clifford algebras}

Every Clifford algebra $\cl{p,q}$, $s_8=(p-q) \text{ mod } 8$, is isomorphic to one of the following (square) matrix algebras\footnote{Compare chapter 16 on \textit{matrix representations and periodicity of 8}, as well as Table 1 on p. 217 of \cite{PL:CAandSpin}.} $\M(2d,\R)$, $\M(d,\H)$, $\M(2d,\R^2)$, $\M(d,\H^2)$ or $\M(2d,\C)$. The first argument of $\M$ is the dimension, the second the associated ring\footnote{Associated ring means, that the matrix elements are from the respective ring $\R$, $\R^2$, $\C$, $\H$ or $\H^2$.} $\R$ for $s_8=0,2$, $\R^2$ for $s_8=1$, $\C$ for $s_8=3,7$, $\H$ for $s_8=4,6$, and $\H^2$ for $s_8=5$. For even $n$: $d=2^{(n-2)/2}$, for odd $n$: $d=2^{(n-3)/2}$.

It has been shown \cite{HA:GeoRoots-1,HHA:ICCA9} that $Sc(f) = 0$ for every square root of $-\unitm$ in every matrix algebra $\A$ isomorphic to $\cl{p,q}$. One can distinguish \textit{ordinary} square roots of $-\unitm$, and \textit{exceptional} ones. All square roots of $-\unitm$ in $\cl{p,q}$ can be computed using the package CLIFFORD for Maple \cite{asvd,AF:CLIFFORD,worksheets,Maple}. 

In all cases the \textit{ordinary} square roots $f$ of $-\unitm$ constitute a \textit{unique conjugacy class} of dimension $\dim(\A)/2$, which has \textit{as many connected components as the group} $\G(\A)$ of invertible elements in $\A$. Furthermore, for ordinary square roots of $-\unitm$ we always have $\Spec(f) = 0$ (zero pseudoscalar part) if the associated ring is $\R^2$, $\HQ^2$, or $\C$. The exceptional square roots of $-\unitm$ \textit{only} exist if $\A \cong \M(2d,\C)$. 

For $\A=\M(2d,\R)$, the centralizer (set of all elements in $\cl{p,q}$ commuting with $f$) and the conjugacy class of a square root $f$ of $-\unitm$ both have $\R$-dimension $2d^2$ with \textit{two connected components}. For the simplest case $d = 1$ we have the algebra $\cl{2,0}$ isomorphic to
$\M(2,\R)$, see the left side of Fig. \ref{fg:Cln=2}.

For $\A=\M(2d,\R^2)=\M(2d,\R)\times\M(2d,\R)$, the square roots of $(-\unitm,-\unitm)$ are pairs of two square roots of $-\unitm$ in $\M(2d,\R)$. They constitute a unique conjugacy class with \textit{four connected components}, each of dimension $4d^2$. Regarding the four connected components, the group of inner automorphisms $\Inn(\A)$ induces the permutations of the Klein group, whereas the quotient group $\Aut(\A)/\Inn(\A)$ is isomorphic to the group of isometries of a Euclidean square in 2D. The simplest example with $d=1$ is $\cl{2,1}$ isomorphic to $M(2,\R^2)=\M(2,\R)\times\M(2,\R)$.

For $\A=\M(d,\HQ)$, the submanifold of the square roots $f$ of $-\unitm$ is a \textit{single connected conjugacy class} of $\R$-dimension $2d^2$ equal to the $\R$-dimension of the centralizer of every $f$. The easiest example for $d=1$ is $\HQ$, isomorphic to $\cl{0,2}$, see the right side of Fig. \ref{fg:Cln=2}.

\begin{figure}
  \begin{center}      
    \includegraphics[scale=0.35]{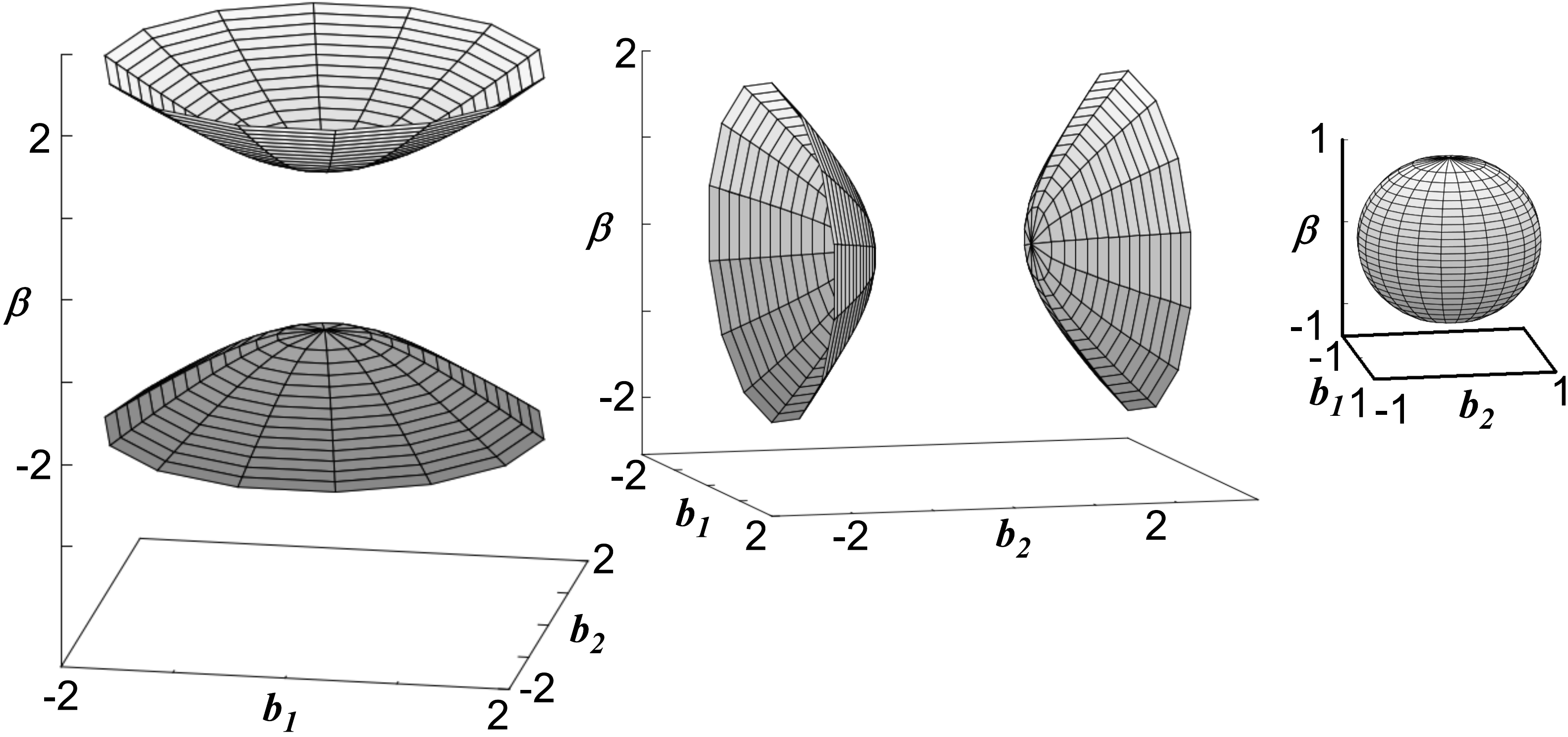}
    \caption{Manifolds of square roots $f$ of $-1$ in $\cl{2,0}$ (left), $\cl{1,1}$ (center), and $\cl{0,2}\cong\H$ (right). The square roots are $f=\alpha + b_1 e_1+b_2e_2+\beta e_{12},$ with $\alpha, b_1, b_2, \beta \in \R$, $\alpha=0$, and $\beta^2=b_1^2e_2^2+b_2^2e_1^2+e_1^2e_2^2$.}  
    \label{fg:Cln=2}
  \end{center}
\end{figure}

For $\A=\M(d,\HQ^2)=\M(d,\HQ)\times\M(d,\HQ)$, the square roots of $(-\unitm,-\unitm)$ are pairs of two square roots $(f,f')$ of $-\unitm$ in $\M(d,\HQ)$ and constitute a \textit{unique connected conjugacy class} of $\R$-dimension $4d^2$. The group $\Aut(\A)$ has two connected components: the neutral component $\Inn(\A)$ connected to the identity and the second component containing the swap automorphism $(f,f')\mapsto (f',f)$. The simplest case for $d=1$ is $\HQ^2$ isomorphic to $\cl{0,3}$.

For $\A=\M(2d,\C)$, the square roots of $-\unitm$ are in \textit{bijection to the idempotents} \cite{AFPR:idem}. First, the \textit{ordinary} square roots of $-\unitm$ (with $k=0$) constitute a conjugacy class of $\R$-dimension $4d^2$ of a \textit{single connected component} which is invariant under $\Aut(\A)$. Second, there are $2d$ \textit{conjugacy classes} of \textit{exceptional} square roots of $-\unitm$, each composed of a \textit{single connected component}, characterized by the equality $\Spec(f) = k/d$ (the pseudoscalar coefficient) with $\pm k \in \{1, 2, \ldots, d\}$, and their $\R$-dimensions are $4(d^2-k^2)$. The group $\Aut(\A)$ includes conjugation of the pseudoscalar $\omega \mapsto -\omega$ which maps the conjugacy class associated with $k$ to the class associated with $-k$. The simplest case for $d=1$ is the Pauli matrix algebra isomorphic to the geometric algebra $\cl{3,0}$ of 3D Euclidean space $\R^3$, and to complex biquaternions \cite{SJS:Biqroots}.

\section{The $\pm$ split with respect to two square roots of $-1$}
\label{sc:pmsplit}

With respect to any square root $f\in \cl{p,q}$ of $-1$, $f^2=-1$, every multivector $A\in \cl{p,q}$ can be split into \textit{commuting} and \textit{anticommuting} parts~\cite{HHA:ICCA9}.
\begin{lem}
\label{lm:fsplit}
Every multivector $A \in \cl{p,q}$ has, with respect to a square root $f\in \cl{p,q}$ of~$-1$, i.e., $f^{-1}=-f,$ the unique decomposition
\begin{gather} 
  A_{+f} = \frac{1}{2}(A + f^{-1}Af),\qquad  
  A_{-f} = \frac{1}{2}(A - f^{-1}Af)
  \nonumber \\
  A = A_{+f}+A_{-f}, 
  \qquad A_{+f}\,f = f A_{+f}, 
  \qquad A_{-f}\,f = -fA_{-f}. 
\end{gather}
\end{lem}

For $f,g \in \cl{p,q}$ an arbitrary pair of square roots of $-1$, $f^2=g^2=-1$, the map $f\sandwich g$ is an involution, because $f^2 x g^2 = (-1)^2x = x, \forall \, x\in \cl{p,q}$. In \cite{EH:QFTgen} a split of quaternions by means of the pure unit quaternion basis elements $\vect{i}, \vect{j} \in \HQ$ was introduced, and generalized to a general pair of pure unit quaternions in \cite{EH:OPS-QFT}. We now \textit{generalize the split} to $\cl{p,q}$.

\begin{defn}[$\pm$ split with respect to two square roots of $-1$]
\label{df:genOPS}
Let $f,g$ $\in$ $\cl{p,q}$ be an arbitrary pair of square roots of $-1$, $f^2=g^2=-1$, including the cases $f = \pm g$. The general $\pm$ split is then defined with respect to the two square roots $f, g$ of $-1$ as 
\begin{equation}
  \label{eq:genOPS}
  x_{\pm} = \frac{1}{2}(x \pm f x g), \qquad \forall\, x\in \cl{p,q}.
\end{equation} 
\end{defn}

Note that the split of Lemma \ref{lm:fsplit} is a special case of Definition \ref{df:genOPS} with $g=-f$. 

We observe from (\ref{eq:genOPS}), that $f x g = x_+ - x_-$, i.e. under the map $f\sandwich g$ the $x_+$ part is invariant, but the $x_-$ part changes sign
\begin{equation}
  fx_{\pm}g 
  = \frac{1}{2}(fxg \pm f^2 x g^2)
  = \frac{1}{2}(fxg \pm x )
  = \pm \frac{1}{2}(x \pm f x g) 
  = \pm x_{\pm}.
  \label{eq:fgrotqm}
\end{equation} 

The two parts $x_{\pm}$ can be represented with Lemma \ref{lm:fsplit} as linear combinations of $x_{+f}$ and $x_{-f}$, or of $x_{+g}$ and $x_{-g}$
\begin{align} 
  x_{\pm} 
  &= \frac{1}{2}(x_{+f}+x_{-f} \pm f (x_{+f}+x_{-f}) g)
  = x_{+f}\,\frac{1\pm fg}{2} + x_{-f}\,\frac{1\mp fg}{2}
  \nonumber \\
  &= \frac{1}{2}(x_{+g}+x_{-g} \pm f (x_{+g}+x_{-g}) g)
  = \frac{1\pm fg}{2}\,x_{+g} + \frac{1\mp fg}{2}\,x_{-g} .
\end{align} 

For $\cl{p,q}\cong \M(2d,\C)$ or $\M(d,\H)$ or $\M(d,\H^2)$, or for both $f,g$ being blades in 
$Cl(p,q)\cong \M(2d,\R)$ or $\M(2d,\R^2)$, we have $\widetilde{f}=-f$, $\widetilde{g}=-g$. We therefore obtain the following lemma. 

\begin{lem}[Orthogonality of two $\pm$ split parts]\label{lm:OPSortho}
  Assume in $\cl{p,q}$ two square roots $f,g$ of $-1$ with $\widetilde{f}=-f$, $\widetilde{g}=-g$.
  Given any two multivectors $x,y \in \cl{p,q}$ and applying the $\pm$ split (\ref{eq:genOPS}) with respect to $f,g$ we get zero for the scalar part of the mixed products
  \begin{equation}
     Sc(x_+\widetilde{y_-}) = 0, 
     \qquad Sc(x_-\widetilde{y_+}) = 0 .
  \end{equation}
\end{lem}
We prove the first identity, the second follows from $Sc(x) = Sc(\widetilde{x})$. 
\begin{align}
  Sc(x_+\widetilde{y_-})
  &= \frac{1}{4}Sc((x+fxg)(\widetilde{y}-g\widetilde{y}f))
  = \frac{1}{4}Sc(x\widetilde{y}-fxgg\widetilde{y}f+fxg\widetilde{y}-xg\widetilde{y}f)
  \nonumber \\
  &= 
    \frac{1}{4}Sc(x\widetilde{y}-x\widetilde{y}+xg\widetilde{y}f - xg\widetilde{y}f) = 0,
  \label{eq:prooffgortho}
\end{align}
where the symmetry $Sc(xy)=Sc(yx)$ was used for the third equality.

We will now establish the \textit{general identity}
\be 
  e^{\alpha f} x_{\pm} e^{\beta g} 
  = x_{\pm} e^{(\beta\mp\alpha) g}
  = e^{(\alpha \mp \beta)f } x_{\pm} .
  \label{eq:expqexp}
\end{equation} 
First, we prove
\begin{align}
  &e^{\alpha f} x_{+} 
  = (\cos \alpha + f \sin \alpha ) \frac{1}{2}(x + fxg)
  = \frac{1}{2}(x + fxg) \cos \alpha
    + \frac{1}{2} f (x + fxg) \sin \alpha
  \nonumber \\
  &= \frac{1}{2}(x + fxg) \cos \alpha
    + \frac{1}{2} (fx g(-g) - xg) \sin \alpha
  \label{eq:fx+}  
  \\
  &= \frac{x + fxg}{2} \cos \alpha
    + \frac{fxg + x}{2} (-g) \sin \alpha
  = \frac{x + fxg}{2} (\cos \alpha -g \sin \alpha)
  = x_{+} e^{-\alpha g}.
  \nonumber 
\end{align}
Similarly we can prove that $e^{\alpha f} x_{-} = x_{-} e^{+\alpha g}$ by replacing $g\rightarrow -g$ $(\Rightarrow x_+\rightarrow x_-)$ in \eqref{eq:fx+}.

\section{General two-sided Clifford Fourier transforms}
\label{sc:Gen2sCFT}

The \textit{general two-sided Clifford Fourier transform} (CFT), to be introduced now, can both be understood as a generalization of known one-sided CFTs \cite{HM:CFToMVF}, or of the two-sided quaternion Fourier transformation (QFT) \cite{EH:QFTgen,EH:OPS-QFT} to a general Clifford algebra setting. Most previously known CFTs use in their kernels specific square roots of $-1$, like bivectors, pseudoscalars, unit pure quaternions, or sets of coorthogonal blades (commuting or anticommuting blades) \cite{BSH:genGFT}. We will \textit{remove all these restrictions} on the square roots of $-1$ used in a CFT. Note that if the left or right phase angle is identical to zero, we get one-sided right or left sided CFTs, respectively.

\begin{defn}[CFT with respect to two square roots of $-1$]
\label{df:CFTfg}
Let $f,g$ $\in$ $\cl{p,q}$, $f^2=g^2=-1$, be any two square roots of $-1$. The general two-sided Clifford Fourier transform (CFT) of $h\in L^1(\R^{p,q},\cl{p,q})$, with respect to $f,g$ is 
\begin{equation}
  \label{eq:CFTfg}
  \mathcal{F}^{f,g}\{ h \}(\bomega)
  = \int_{\R^{p,q}} e^{-f u(\svect{x},\bomega)} h(\bvect{x}) 
              \,e^{-g v(\svect{x},\bomega)} d^n\bvect{x},
\end{equation}
where $d^n\bvect{x} = dx_1\ldots dx_n$, $\bvect{x}, \bomega \in \R^{p,q}$, and $u,v: \R^{p,q}\times\R^{p,q}\rightarrow \R $.
\end{defn}

Since square roots of $-1$ in $\cl{p,q}$ populate \textit{continuous submanifolds} in $\cl{p,q}$, the CFT of Definition \ref{df:CFTfg} is generically \textit{steerable} within these manifolds. In Definition \ref{df:CFTfg}, the two square roots $f,g \in \cl{p,q}$ of $-1$ may be from the same (or different) conjugacy class and component, respectively. 

Linearity of the CFT integral \eqref{eq:CFTfg} allows us to use the $\pm$ split $h=h_- + h_+$ of Definition \ref{df:genOPS} to obtain
\begin{align}
  \mathcal{F}^{f,g}\{ h \}(\bomega)
  &= \mathcal{F}^{f,g}\{ h_- \}(\bomega) + \mathcal{F}^{f,g}\{ h_+ \}(\bomega)
  \nonumber \\
  &= \mathcal{F}^{f,g}_-\{ h \}(\bomega) + \mathcal{F}_+^{f,g}\{ h \}(\bomega),
	\label{eq:CFTsplit}
\end{align} 
since by their construction the operators of the Clifford Fourier transformation $\mathcal{F}^{f,g}$, and of the $\pm$ split with respect to $f,g$ commute. 
From (\ref{eq:expqexp}) follows 

\begin{thm}[CFT of $h_{\pm}$]
\label{th:fpmtrafo}
The CFT of the $\pm$ split parts $h_{\pm}$ , with respect to two square roots $f,g \in \cl{p,q}$ of $-1$ , of a Clifford module function 
$h \in$ $L^1 \left(\R^{p,q}; \right.$ $\left. \cl{p,q}\right)$ have the quasi-complex forms
\begin{align}
  \mathcal{F}^{f,g}_{\pm}\{h\} 
  &= \mathcal{F}^{f,g}\{h_{\pm}\} 
  \nonumber \\
  &= \int_{\R^{p,q}}
    h_{\pm}\,e^{-g (v(\svect{x},\bomega) \mp u(\svect{x},\bomega))}d^nx
  \stackrel{}{=} \int_{\R^{p,q}}
    e^{-f (u(\svect{x},\bomega) \mp v(\svect{x},\bomega))}h_{\pm}\,d^nx \,\, .
\end{align}
\end{thm} 

\begin{rem}
Theorem \ref{th:fpmtrafo} establishes in combination with \eqref{eq:CFTsplit} a general method for how to compute a two-sided CFT in terms of two one-sided CFTs. For special relations of two-sided and one-sided quaternionic Fourier transforms see \cite{asvd,EH:QFTgen,EH:DirQFTUP,EH:OPS-QFT,HS:Orth2dPSplit}. 
\end{rem}

\begin{rem}
The quasi-complex forms in Theorem \ref{th:fpmtrafo} allow to establish \textit{discretized} and \textit{fast} versions of the general two-sided CFT of Definition \ref{df:CFTfg} as sums of complex discretized and fast Fourier transformations (FFT), respectively. 
\end{rem}

\section{Properties of the general two-sided CFT}
\label{sc:CFTprop}

We now study important properties of the general two-sided CFT of Definition \ref{df:CFTfg}.

\subsection{Linearity, shift, modulation, dilation, and powers of $f,g$}

Regarding \textit{left and right linearity} of the general two-sided CFT of Definition \ref{df:CFTfg} we can establish with the help of Lemma \ref{lm:fsplit} the following proposition. 
\begin{prop}[Left and right linearity]  
 For $h_1,h_2\in L^1(\R^{p,q}; \cl{p,q})$, and constants $\alpha,\beta \in \cl{p,q}$ we have
\begin{align} 
  \mathcal{F}^{f,g}\{ \alpha h_1 + \beta h_2 \}(\bomega)
  &= \alpha_{+f}\mathcal{F}^{f,g}\{h_1\}(\bomega) + \alpha_{-f}\mathcal{F}^{-f,g}\{h_1\}(\bomega) 
  \nonumber \\
  \label{eq:leftlin}
  &\phantom{= } + \beta_{+f}\mathcal{F}^{f,g}\{h_2\}(\bomega) + \beta_{-f}\mathcal{F}^{-f,g}\{h_2\}(\bomega),
  \\
  \mathcal{F}^{f,g}\{ h_1 \alpha  + h_2 \beta  \}(\bomega)
  &= \mathcal{F}^{f,g}\{h_1\}(\bomega) \alpha_{+g} + \mathcal{F}^{f,-g}\{h_1\}(\bomega) \alpha_{-g}
  \nonumber \\
  \label{eq:rightlin}
  &\phantom{= } + \mathcal{F}^{f,g}\{h_2\}(\bomega) \beta_{+g} + \mathcal{F}^{f,-g}\{h_2\}(\bomega) \beta_{-g}\, .
\end{align}
\end{prop}

\begin{proof}
Based on Lemma \ref{lm:fsplit} we have
\begin{align} 
  &\alpha = \alpha_{+f} + \alpha_{-f}, \qquad
  \alpha_{+f}f = f\alpha_{+f}, \qquad
  \alpha_{-f}f = -f\alpha_{-f}
  \nonumber \\
  &\Rightarrow \,\,\,
  e^{-fu} \alpha 
  = e^{-fu} (\alpha_{+f} + \alpha_{-f})
  = e^{-fu}\alpha_{+f} + e^{-fu}\alpha_{-f}
  \nonumber \\
  &\phantom{ \Rightarrow \,\,\, e^{-fu} \alpha }\,\,
  = \alpha_{+f}e^{-fu} + \alpha_{-f}e^{-(-f)u},
  \label{eq:falp}
\end{align}
and similarly 
\be
  \beta = \beta_{+f} + \beta_{-f}, \qquad
  e^{-fu} \beta
  = \beta_{+f}e^{-fu} + \beta_{-f}e^{-(-f)u},
  \label{eq:fbet}
\end{equation}
as well as
\be
  \label{eq:galpbet}
  \alpha e^{-gv} 
  = e^{-gv} \alpha_{+g} + e^{-(-g)v}\alpha_{-g}, \qquad
  \beta e^{-gv} 
  = e^{-gv} \beta_{+g} + e^{-(-g)v}\beta_{-g},
\end{equation} 
We insert \eqref{eq:falp} and \eqref{eq:fbet} into Definition \ref{df:CFTfg} and get
\begin{align} 
  &\mathcal{F}^{f,g}\{ \alpha h_1 + \beta h_2 \}(\bomega)
  = \int_{\R^{p,q}} e^{-f u} \{\alpha h_1 + \beta h_2 \}
              \,e^{-g v} d^n\bvect{x}
  \nonumber \\
  &= \int_{\R^{p,q}} 
    \{\alpha_{+f}e^{-fu}h_1 + \alpha_{-f}e^{-(-f)u} h_1 
    + \beta_{+f}e^{-fu}h_1 + \beta_{-f}e^{-(-f)u} h_2 \}
    \,e^{-g v} d^n\bvect{x}
  \nonumber \\
  &= \alpha_{+f}\mathcal{F}^{f,g}\{h_1\}(\bomega) 
     + \alpha_{-f}\mathcal{F}^{-f,g}\{h_1\}(\bomega)
  \nonumber \\
  &\phantom{= }  \,\,
     + \beta_{+f}\mathcal{F}^{f,g}\{h_2\}(\bomega) 
     + \beta_{-f}\mathcal{F}^{-f,g}\{h_2\}(\bomega).
\end{align}
By instead applying \eqref{eq:galpbet} we can similarly derive \eqref{eq:rightlin}.
\end{proof}

Regarding the CFT of $\bvect{x}$-\textit{shifted} functions we obtain the following proposition.
\begin{prop}[$\bvect{x}$-{shift}]
For an $\bvect{x}$-\textit{shifted} function $h_0(\bvect{x})= h(\bvect{x}-\bvect{x}_0)$, $h \in L^1(\R^{p,q}; \cl{p,q})$, with constant $\bvect{x}_0\in \R^{p,q}$, assuming linearity of $u(\vect{x},\bomega), v(\vect{x},\bomega)$ in their vector space argument $\vect{x}$, we get
\be 
  \mathcal{F}^{f,g}\{ h_0 \}(\bomega)
  = e^{-f u(\svect{x}_0,\bomega)} \mathcal{F}^{f,g}\{ h \}(\bomega) 
  \,e^{-g v(\svect{x}_0,\bomega)}.
\end{equation} 
\end{prop}
\begin{proof}
We assume linearity of $u(\vect{x},\bomega), v(\vect{x},\bomega)$ in their vector space argument $\vect{x}$. Inserting $h_0(\bvect{x})= h(\bvect{x}-\bvect{x}_0)$ in Definition \ref{df:CFTfg} we obtain
\begin{align}
  \mathcal{F}^{f,g}\{ h_0 \}(\bomega)
  &= \int_{\R^{p,q}} e^{-f u(\svect{x},\bomega)} h(\bvect{x}-\bvect{x}_0)
              \,e^{-g v(\svect{x},\bomega)} d^n\bvect{x}
  \nonumber \\
  &= \int_{\R^{p,q}} e^{-f u(\svect{y}+\svect{x}_0,\bomega)} h(\bvect{y})
              \,e^{-g v(\svect{y}+\svect{x}_0,\bomega)} d^n\bvect{y}
  \nonumber \\
  &= \int_{\R^{p,q}} e^{-f u(\svect{x}_0,\bomega)}
                    e^{-f u(\svect{y},\bomega)} h(\bvect{y})
              \,e^{-g v(\svect{y},\bomega)}
                e^{-g v(\svect{x}_0,\bomega)} d^n\bvect{y}
  \nonumber \\
  &= e^{-f u(\svect{x}_0,\bomega)}
    \int_{\R^{p,q}} e^{-f u(\svect{y},\bomega)} h(\bvect{y})
              \,e^{-g v(\svect{y},\bomega)} d^n\bvect{y}
    \,e^{-g v(\svect{x}_0,\bomega)}
  \nonumber \\
  &= e^{-f u(\svect{x}_0,\bomega)}
    \mathcal{F}^{f,g}\{ h \}(\bomega)
    \,e^{-g v(\svect{x}_0,\bomega)},
\end{align}
where we have substituted $\bvect{y} = \bvect{x}-\bvect{x}_0$ for the second equality, we used the linearity of $u(\vect{x},\bomega)$ and $v(\vect{x},\bomega)$ in their vector space argument $\vect{x}$ for the third equality, and that $e^{-f u(\svect{x}_0,\bomega)}$ and $e^{-g v(\svect{x}_0,\bomega)}$ are independent of $\bvect{y}$ for the fourth equality. 
\end{proof}

The next proposition on the CFT of \textit{modulated} functions assumes special linearity properties of the functions $u,v$.
\begin{prop}[Modulation] 
Assume that the functions $u(\bvect{x},\bomega), v(\bvect{x},\bomega)$ are both linear in their frequency argument $\bomega$. Then we obtain for $h_m(\bvect{x})=e^{-f u(\svect{x},\bomega_0)}h(x)\,e^{-g v(\svect{x},\bomega_0)}$, $h \in L^1(\R^{p,q}; \cl{p,q})$, and constant $\bomega_0\in \R^{p,q}$ the modulation formula
\be 
  \mathcal{F}^{f,g}\{ h_m \}(\bomega)
  = \mathcal{F}^{f,g}\{ h \}(\bomega+\bomega_0). 
\end{equation} 
\end{prop}
\begin{proof}
We assume, that the functions $u(\bvect{x},\bomega)$ and $v(\bvect{x},\bomega)$ are both linear in their frequency argument $\bomega$. Inserting $h_m(\bvect{x})=e^{-f u(\svect{x},\bomega_0)}h(x)\,e^{-g v(\svect{x},\bomega_0)}$ in Definition \ref{df:CFTfg} we obtain
\begin{align}
  \mathcal{F}^{f,g}\{ h_m \}(\bomega)
  &= \int_{\R^{p,q}} e^{-f u(\svect{x},\bomega)} h_m(\bvect{x})
              \,e^{-g v(\svect{x},\bomega)} d^n\bvect{x}
  \nonumber \\
  &= \int_{\R^{p,q}} e^{-f u(\svect{x},\bomega)} e^{-f u(\svect{x},\bomega_0)}h(x)\,e^{-g v(\svect{x},\bomega_0)}
              \,e^{-g v(\svect{x},\bomega)} d^n\bvect{x}
  \nonumber \\
  &= \int_{\R^{p,q}} e^{-f u(\svect{x},\bomega+\bomega_0)} h(x)
              \,e^{-g v(\svect{x},\bomega+\bomega_0)} d^n\bvect{x}
  \nonumber \\
  &= \mathcal{F}^{f,g}\{ h \}(\bomega+\bomega_0),
\end{align}
where we used the linearity of $u(\bvect{x},\bomega)$ and $v(\bvect{x},\bomega)$ in their frequency argument $\bomega$ for the third equality. 
\end{proof}

Regarding \textit{dilations}, further special assumption are made for the functions $u,v$.
\begin{prop}[Dilations]
\label{prp:dil}
Assume that for constants $a_1, \ldots, a_n$ $\in \R\setminus\{0\}$, and $\bvect{x}' = \sum_{k=1}^n a_kx_k\bvect{e}_k$, we have $u(\bvect{x}',\bomega) = u(\bvect{x},\bomega')$, and $v(\bvect{x}',\bomega) = v(\bvect{x},\bomega')$, with $\bomega' = \sum_{k=1}^n a_k\omega_k\bvect{e}_k$. We then obtain for $h_d(\bvect{x})=h(\bvect{x}')$, $h \in L^1(\R^{p,q}; \cl{p,q})$, that
\be 
  \mathcal{F}^{f,g}\{ h_d \}(\bomega)
  = \frac{1}{|a_1 \ldots a_n|}\mathcal{F}^{f,g}\{ h \}(\bomega_d), \qquad
  \bomega_d = \sum_{k=1}^n \frac{1}{a_k}\omega_k\bvect{e}_k . 
\end{equation} 
\end{prop}
\begin{proof}
We assume for constants $a_1, \ldots, a_n$ $\in \R\setminus\{0\}$, and $\bvect{x}' = \sum_{k=1}^n a_kx_k\bvect{e}_k$, that we have $u(\bvect{x}',\bomega) = u(\bvect{x},\bomega')$, and $v(\bvect{x}',\bomega) = v(\bvect{x},\bomega')$, with $\bomega' = \sum_{k=1}^n a_k\omega_k\bvect{e}_k$. Inserting $h_d(\bvect{x})=h(\bvect{x}')$ in Definition \ref{df:CFTfg} we obtain
\begin{align}
  \mathcal{F}^{f,g}\{ h_d \}(\bomega)
  &= \int_{\R^{p,q}} e^{-f u(\svect{x},\bomega)} h_d(\bvect{x})
              \,e^{-g v(\svect{x},\bomega)} d^n\bvect{x}
  \nonumber \\
  &= \int_{\R^{p,q}} e^{-f u(\svect{x},\bomega)} h(\bvect{x}')
              \,e^{-g v(\svect{x},\bomega)} d^n\bvect{x}
  \nonumber \\
  &= \frac{1}{|a_1 \ldots a_n|}\int_{\R^{p,q}} e^{-f u(\svect{y}^{\backprime},\bomega)} h(\bvect{y})
              \,e^{-g v(\svect{y}^{\backprime},\bomega)} d^n\bvect{y}
  \nonumber \\
  &= \frac{1}{|a_1 \ldots a_n|}\int_{\R^{p,q}} e^{-f u(\svect{y},\bomega_d)} h(\bvect{y})
              \,e^{-g v(\svect{y},\bomega_d)} d^n\bvect{y}
  \nonumber \\
  &= \frac{1}{|a_1 \ldots a_n|}\mathcal{F}^{f,g}\{ h \}(\bomega_d),              
\end{align}
where we substituted $\bvect{y}=\bvect{x}'=\sum_{k=1}^n a_kx_k\bvect{e}_k$ and  $\bvect{x} = \sum_{k=1}^n \frac{1}{a_k}y_k\bvect{e}_k = \bvect{y}^{\backprime}$ for the third equality. Note that in this step each negative $a_k<0, 1\leq k \leq n$, leads to a factor $\frac{1}{|a_k|}$, because the negative sign is absorbed by interchanging the resulting integration boundaries $\{+\infty, -\infty\}$ to $\{-\infty, +\infty\}$. For the fourth equality we applied the assumption $u(\bvect{y}^{\backprime},\bomega) = u(\bvect{y},\bomega^{\backprime})$, $v(\bvect{y}^{\backprime},\bomega) = v(\bvect{y},\bomega^{\backprime})$ and defined $\bomega_d=\bomega^{\backprime}=\sum_{k=1}^n \frac{1}{a_k}\omega_k\bvect{e}_k$. 
\end{proof}

\begin{cor}[Isotropic dilation]
\label{cor:isodil}
For $a_1= \ldots = a_n = a \in \R\setminus\{0\}$ Proposition \ref{prp:dil} simplifies under the same special assumption for $u,v$ to 
\be 
  \mathcal{F}^{f,g}\{ h_d \}(\bomega)
  = \frac{1}{|a|^n}\mathcal{F}^{f,g}\{ h \}(\frac{1}{a} \bomega) . 
\end{equation} 
\end{cor}
\noindent
Note, that the above assumption for $u,v$ in Proposition \ref{prp:dil} and Corollary \ref{cor:isodil} would, e.g., be fulfilled for $u(\bvect{x},\bomega) = v(\bvect{x},\bomega) = \bvect{x}\ast\widetilde{\bomega}=\sum_{k=1}^n x_k\omega_k$. 

\begin{prop}[Power factors]
For $f,g$ power factors in $h_{p,q}(\bvect{x}) = f^p h(\bvect{x}) g^q$, $p,q \in \Z$, $h \in L^1(\R^{p,q}; \cl{p,q})$, we obtain
\be 
  \mathcal{F}^{f,g}\{ h_{p,q} \}(\bomega)
  = f^p \mathcal{F}^{f,g}\{ h \}(\bomega) g^q .
\end{equation} 
\end{prop}
\begin{proof}
By direct computation we find
\begin{align}
  \mathcal{F}^{f,g}\{ h_{p,q} \}(\bomega)
  &= \int_{\R^{p,q}} e^{-f u(\svect{x},\bomega)} h_{p,q}(\bvect{x})
              \,e^{-g v(\svect{x},\bomega)} d^n\bvect{x}
  \nonumber \\
  &= \int_{\R^{p,q}} e^{-f u(\svect{x},\bomega)} f^p h(\bvect{x})\, g^q
              \,e^{-g v(\svect{x},\bomega)} d^n\bvect{x}
  \nonumber \\
  &= f^p \int_{\R^{p,q}} e^{-f u(\svect{x},\bomega)} h(\bvect{x})
              \,e^{-g v(\svect{x},\bomega)} d^n\bvect{x} \,g^q
  \nonumber \\
  &= f^p \mathcal{F}^{f,g}\{ h \}(\bomega) g^q ,
\end{align}
where we have used $e^{-f u(\svect{x},\bomega)} f^p=f^p e^{-f u(\svect{x},\bomega)}$ and $g^q\,e^{-g v(\svect{x},\bomega)} = e^{-g v(\svect{x},\bomega)}g^q$ for the third equality, which is obvious from $e^{-f u} = \sum_{k=1}^{\infty}\frac{(-fu)^k}{k!}$ and $e^{-g v} = \sum_{k=1}^{\infty}\frac{(-gv)^k}{k!}$ with $u,v \in \R$. As an alternative to this proof, we could also apply \eqref{eq:leftlin} with $\alpha = f^p, \beta = 0$ followed by \eqref{eq:rightlin} with $\alpha = g^p, \beta = 0$. 
\end{proof}

\subsection{CFT inversion, moments, derivatives, Plancherel, Parseval}

For establishing an inversion formula, moment and derivative properties, Plancherel and Parseval identities, certain \textit{assumptions} about the phase functions $u(\bvect{x},\bomega)$, $v(\bvect{x},\bomega)$ need to be made. One possibility is, e.g. to arbitrarily partition the scalar product $\bvect{x}\ast\widetilde{\bomega} = \sum_{l=1}^n x_l\omega_l = u(\bvect{x},\bomega) + v(\bvect{x},\bomega)$, with 
\be 
  u(\bvect{x},\bomega) = \sum_{l=1}^k x_l\omega_l, \qquad
  v(\bvect{x},\bomega) = \sum_{l=k+1}^n x_l\omega_l,
  \label{eq:assumeuv}
\end{equation} 
for any arbitrary but fixed $1 \leq k \leq n$. We could also include any subset $A_u \subseteq \{1, \ldots , n\}$ of coordinates in $u(\bvect{x},\bomega)$ and the complementary set $A_v = \{1, \ldots , n\}\setminus A_u$ of coordinates in $v(\bvect{x},\bomega)$, etc. \eqref{eq:assumeuv} will be assumed for the current subsection.

We then get the following \textit{inversion} theorem.
\begin{thm}[CFT inversion]
With the assumption \eqref{eq:assumeuv} for $u,v$ we get for $h \in L^1(\R^{p,q}; \cl{p,q})$, that
\begin{equation}
  \label{eq:invCFTfg}
  h(\bvect{x})=\mathcal{F}_{-1}^{f,g}\{ \mathcal{F}^{f,g}\{ h \} \}(\bvect{x})
  = \frac{1}{(2\pi)^n}\int_{\R^{p,q}} e^{f u(\svect{x},\bomega)} \mathcal{F}^{f,g}\{ h \}(\bomega) 
              \,e^{g v(\svect{x},\bomega)} d^n\bomega,
\end{equation}
where $d^n\bomega = d\omega_1\ldots d\omega_n$, $\bvect{x}, \bomega \in \R^{p,q}$. For the existence of (\ref{eq:invCFTfg}) we further need $\mathcal{F}^{f,g}\{ h \}\in L^1(\R^{p,q}; \cl{p,q})$.
\end{thm}
\begin{proof}
By direct computation we find
\begin{align}
  &\frac{1}{(2\pi)^n}\int_{\R^{p,q}} e^{f u(\svect{x},\bomega)} 
    \mathcal{F}^{f,g}\{ h \}(\bomega)\,e^{g v(\svect{x},\bomega)} d^n\bomega
  \nonumber \\
  &=\frac{1}{(2\pi)^n}\int_{\R^{p,q}} \int_{\R^{p,q}} 
    e^{f u(\svect{x},\bomega)} e^{-f u(\svect{y},\bomega)}
    h(\bvect{y})
    \,e^{-g v(\svect{y},\bomega)} e^{g v(\svect{x},\bomega)}d^n\bvect{y}
    d^n\bomega
  \nonumber \\
  &=\frac{1}{(2\pi)^n}\int_{\R^{p,q}} \int_{\R^{p,q}} 
    e^{f u(\svect{x}-\svect{y},\bomega)} 
    h(\bvect{y})
    \,e^{g v(\svect{x}-\svect{y},\bomega)}
    d^n\bomega d^n\bvect{y}
  \nonumber \\
  &= \frac{1}{(2\pi)^n}\int_{\R^{p,q}}\int_{\R^{p,q}} e^{f \sum_{l=1}^k (x_l-y_l)\omega_l}
     h(\bvect{y})\,e^{g \sum_{m=k+1}^n (x_m-y_m)\omega_m}
     d^n\bomega d^n\bvect{y}
  \nonumber \\
  &= \frac{1}{(2\pi)^n}\int_{\R^{p,q}} \int_{\R^{p,q}}
     \prod_{l=1}^k e^{f (x_l-y_l)\omega_l}
     h(\bvect{y})\,\prod_{m=k+1}^n e^{g (x_m-y_m)\omega_m}
     d^n\bomega d^n\bvect{y}
  \nonumber \\
  &= \int_{\R^{p,q}} \prod_{l=1}^k \delta(x_l-y_l)\,
     h(\bvect{y})\,\prod_{m=k+1}^n \delta(x_m-y_m)
     d^n\bvect{y} 
  \nonumber \\
  &= h(\bvect{x}), 
\end{align}
where we have inserted Definition \ref{df:CFTfg} for the first equality, used the linearity of $u$ and $v$ according to \eqref{eq:assumeuv} for the second equality, as well as inserted \eqref{eq:assumeuv} for the third equality, and that $\frac{1}{2\pi}\int_{\R}e^{f (x_l-y_l)\omega_l}d\omega_l = \delta(x_l-y_l)$, $1\leq l \leq k$, and $\frac{1}{2\pi}\int_{\R}e^{g (x_m-y_m)\omega_m}d\omega_m = \delta(x_m-y_m)$, $k+1\leq m \leq n$, for the fifth equality. 
\end{proof}
\begin{rem}
An alternative proof for the CFT inversion can be constructed similar to the proof of Theorem 7 in \cite{BBSS:ConvHypCpxFT}, where a Hermite basis for $\mathcal{S}(\R^m)\otimes \cl{0,m}$ was used, with the Schwartz space $\mathcal{S}(\R^m)$. 
\end{rem}

Additionally, we get the transformation law for \textit{partial derivatives} in the following proposition.
\begin{prop}[Partial derivatives]
\label{prop:partd}
For $h'_l(\bvect{x})=\partial_{x_l}h(\bvect{x})$, $1 \leq l \leq n$, $h$ piecewise smooth and integrable, and $h, h'_l \in$ $L^1\left(\R^{p,q};\right.$ $\left. \cl{p,q}\right)$ we obtain
\be 
  \mathcal{F}^{f,g}\{ h'_l \}(\bomega)
  = 
  \left\{ 
  \begin{array}{ll}
  f \omega_l \,\mathcal{F}^{f,g}\{ h \}(\bomega),   & \quad \text{for}\,\,\, l\leq k \\
  \mathcal{F}^{f,g}\{ h \}(\bomega) \,g \,\omega_l, & \quad \text{for}\,\,\, l>k 
  \end{array}
  \right.
  .
\end{equation} 
\end{prop}
\begin{proof}
Assume $l\leq k$. Then we have
\begin{align}
  \mathcal{F}^{f,g}\{ h'_l \}(\bomega)
  &= \int_{\R^{p,q}} e^{-f u(\svect{x},\bomega)}
    h'_l(\bvect{x}) \,e^{-g v(\svect{y},\bomega)} d^n\bvect{x}
  \nonumber \\
  &= \int_{\R^{p,q}} e^{-f u(\svect{x},\bomega)}
    \partial_{x_l}h(\bvect{x}) \,e^{-g v(\svect{y},\bomega)} d^n\bvect{x}
  \nonumber \\
  &= \int_{\R^{p,q}} e^{-f \sum_{l=1}^k x_l\omega_l}
    \partial_{x_l}h(\bvect{x}) \,e^{-g v(\svect{y},\bomega)} d^n\bvect{x}
  \nonumber \\
  &= -\int_{\R^{p,q}} \partial_{x_l}\left( e^{-f \sum_{l=1}^k x_l\omega_l} \right)
    h(\bvect{x}) \,e^{-g v(\svect{y},\bomega)} d^n\bvect{x}
  \nonumber \\
  &= -(-f\omega_l)\int_{\R^{p,q}} e^{-f \sum_{l=1}^k x_l\omega_l}
    h(\bvect{x}) \,e^{-g v(\svect{y},\bomega)} d^n\bvect{x}
  \nonumber \\
  &= f \omega_l \mathcal{F}^{f,g}\{ h \}(\bomega),
\end{align}
where we inserted $u$ of \eqref{eq:assumeuv} for the third equality and performed integration by parts for the fourth equality. For $l > k$ the proof can similarly be done by insertion of $v$ of \eqref{eq:assumeuv} followed by partial integration.
\end{proof}

The CFT of \textit{spatial moments} is \textit{dual} to the transformation property of partial derivatives of Proposition \ref{prop:partd}.
\begin{prop}[Spatial moments]
The CFT of {spatial moments} with $h_l(\bvect{x})={x_l}h(\bvect{x})$, $1 \leq l \leq n$, $h, h_l \in L^1(\R^{p,q}; \cl{p,q})$, results in
\be 
  \label{eq:spacemom}
  \mathcal{F}^{f,g}\{ h_l \}(\bomega)
  = 
  \left\{ 
  \begin{array}{ll}
  f \,\partial_{\omega_l} \,\mathcal{F}^{f,g}\{ h \}(\bomega),  & \quad \text{for}\,\,\, l \leq k \\
  \partial_{\omega_l} \,\mathcal{F}^{f,g}\{ h \}(\bomega) \,g , & \quad \text{for}\,\,\, l > k 
  \end{array}
  \right.
  .
\end{equation}
\end{prop}
\begin{proof}
Assume $l\leq k$. We compute
\begin{align} 
  -f h_l(\bvect{x}) 
  &= -f x_l h(\bvect{x}) 
  = -f x_l \mathcal{F}_{-1}^{f,g}\{ \mathcal{F}^{f,g}\{ h \} \}(\bvect{x})
  \nonumber \\
  &= -f x_l \frac{1}{(2\pi)^n}\int_{\R^{p,q}} e^{f u(\svect{x},\bomega)}  
     \mathcal{F}^{f,g}\{ h \}(\bomega) \,e^{g v(\svect{x},\bomega)} d^n\bomega
  \nonumber \\
  &= -\frac{1}{(2\pi)^n}\int_{\R^{p,q}} f x_l e^{f \sum_{l=1}^k x_l\omega_l}  
     \mathcal{F}^{f,g}\{ h \}(\bomega) \,e^{g v(\svect{x},\bomega)} d^n\bomega
  \nonumber \\
  &= -\frac{1}{(2\pi)^n}\int_{\R^{p,q}} 
     \partial_{\omega_l} \left( e^{f \sum_{l=1}^k x_l\omega_l} \right) 
     \mathcal{F}^{f,g}\{ h \}(\bomega) \,e^{g v(\svect{x},\bomega)} d^n\bomega
  \nonumber \\
  &= \frac{1}{(2\pi)^n}\int_{\R^{p,q}} 
     e^{f \sum_{l=1}^k x_l\omega_l} 
     \left[\partial_{\omega_l} \mathcal{F}^{f,g}\{ h \}(\bomega)\right]
     \,e^{g v(\svect{x},\bomega)} d^n\bomega
  \nonumber \\
  &= \mathcal{F}_{-1}^{f,g} 
     \left[\partial_{\omega_l} \mathcal{F}^{f,g}\{ h \}\right](\bvect{x}),
  \label{eq:proofsmom}     
\end{align}
where we used the inversion formula \eqref{eq:invCFTfg} for the second equality, integration by parts for the sixth equality, and \eqref{eq:invCFTfg} again for the seventh equality. Moreover, by applying the CFT $\mathcal{F}^{f,g}$ to both sides of \eqref{eq:proofsmom} we finally obtain
\be 
   \mathcal{F}^{f,g}\{-f h_l\}(\bomega) 
   = \partial_{\omega_l} \mathcal{F}^{f,g}\{ h \}(\bomega)
   \,\,\, \Leftrightarrow \,\,\,
   \mathcal{F}^{f,g}\{h_l\}(\bomega) 
   = f\partial_{\omega_l} \mathcal{F}^{f,g}\{ h \}(\bomega),
\end{equation} 
because $\mathcal{F}^{f,g}\{-f h_l\}=-f\mathcal{F}^{f,g}\{h_l\}$. We can prove \eqref{eq:spacemom} for $l > k$ analogously. 
\end{proof}

We will next derive both \textit{Plancherel} and \textit{Parseval} identities for the CFT. 
\begin{prop}[Plancherel and Parseval identities]
For the functions $h_1, h_2, h \in L^2(\R^{p,q}; \cl{p,q})$, and assuming\footnote{Remember that in general for $\cl{p,q}\cong \M(2d,\C)$ or $\M(d,\H)$ or $\M(d,\H^2)$, or for both $f,g$ being blades in $\cl{p,q}\cong \M(2d,\R)$ or $\M(2d,\R^2)$, we always have $\widetilde{f}=-f$, $\widetilde{g}=-g$.} that $\widetilde{f}=-f$, $\widetilde{g}=-g$, we obtain the \textit{Plancherel} identity 
\be 
  \label{eq:Planch}
  \langle h_1, h_2 \rangle
  = \frac{1}{(2\pi)^n}\langle \mathcal{F}^{f,g}\{ h_1 \}, \mathcal{F}^{f,g}\{ h_2 \} \rangle,
\end{equation} 
as well as the \textit{Parseval} identity
\be 
  \left\|h\right\|
  = \frac{1}{(2\pi)^{n/2}}\left\| \mathcal{F}^{f,g}\{ h \} \right\| .
\end{equation} 
\end{prop}
\begin{proof}
We only need to prove the Plancherel identity, because the Parseval identity follows from it by setting $h_1=h_2=h$ and by taking the square root on both sides. Assume that $\widetilde{f}=-f$, $\widetilde{g}=-g$. We abbreviate $\int=\int_{\R^{p,q}}$, and compute
\begin{align}
  &\langle \mathcal{F}^{f,g}\{ h_1 \}, \mathcal{F}^{f,g}\{ h_2 \} \rangle
  \nonumber \\
  &= \int \langle \mathcal{F}^{f,g}\{ h_1 \}(\bomega) 
  [\mathcal{F}^{f,g}\{ h_2 \}(\bomega)]^{\sim}\rangle d^n\bomega
  \nonumber \\
  &= \int\int\int \langle 
    e^{-f u(\svect{x},\bomega)} h_1(\bvect{x}) \,e^{-g v(\svect{x},\bomega)}
    d^n\bvect{x}
    [e^{-f u(\svect{y},\bomega)} h_2(\bvect{y}) 
    \,e^{-g v(\svect{y},\bomega)}d^n\bvect{y}]^{\sim}
    \rangle d^n\bomega
  \nonumber \\
  &= \int\int\int \langle 
    e^{-f u(\svect{x},\bomega)} h_1(\bvect{x}) \,e^{-g v(\svect{x},\bomega)}
    e^{-\widetilde{g} v(\svect{y},\bomega)} \widetilde{h_2(\bvect{y})}
    \,e^{-\widetilde{f} u(\svect{y},\bomega)}d^n\bvect{y}  
    \rangle d^n\bvect{x} d^n\bomega
  \nonumber \\
  &= \int\int\int \langle 
    e^{f u(\svect{y},\bomega)}e^{-f u(\svect{x},\bomega)} h_1(\bvect{x}) 
    \,e^{-g v(\svect{x},\bomega)}
    e^{g v(\svect{y},\bomega)} \widetilde{h_2(\bvect{y})}
    \,d^n\bomega \,d^n\bvect{y}  
    \rangle d^n\bvect{x} 
  \nonumber \\
  &= \int\int\int \langle 
    e^{-f u(\svect{x}-\svect{y},\bomega)} h_1(\bvect{x}) 
    \,e^{-g v(\svect{x}-\svect{y},\bomega)}
    \widetilde{h_2(\bvect{y})}
    \,d^n\bomega \,d^n\bvect{y} 
    \rangle d^n\bvect{x} 
  \nonumber \\
  &= (2\pi)^n \int\int\int \langle 
    \frac{e^{-f \sum_{l=1}^k (x_l-y_l)\omega_l}}{(2\pi)^k} h_1(\bvect{x}) 
    \frac{e^{-g \sum_{m=k+1}^n (x_m-y_m)\omega_m}}{(2\pi)^{(n-k)}}\,
    \widetilde{h_2(\bvect{y})}
    \,d^n\bomega \,d^n\bvect{y} 
    \rangle d^n\bvect{x} 
  \nonumber \\
  &= (2\pi)^n \int\int \langle 
    \prod_{l=1}^k\delta(x_l-y_l) \,h_1(\bvect{x}) 
    \,\prod_{m=k+1}^n\delta(x_m-y_m)\,
    \widetilde{h_2(\bvect{y})}
    \,d^n\bvect{y}  
    \rangle d^n\bvect{x}
  \nonumber \\
  &= (2\pi)^n \int \langle 
    h_1(\bvect{x})\widetilde{h_2(\bvect{x})}
    \rangle d^n\bvect{x}
  \nonumber \\ 
  &= (2\pi)^n \langle h_1,h_2 \rangle, 
\end{align}
where we inserted \eqref{eq:symsc} for the first equality, the Definition \ref{df:CFTfg} of the CFT $\mathcal{F}^{f,g}$ for the second equality, applied the principal reverse for the third equality, and the symmetry of the scalar product and that $\widetilde{f}=-f$, $\widetilde{g}=-g$ for the fourth equality, the linearity of $u$ and $v$ according to \eqref{eq:assumeuv} for the fifth equality, inserted the explicit forms of $u$ and $v$ of \eqref{eq:assumeuv} for the sixth equality, and that $\frac{1}{2\pi}\int_{\R}e^{f (x_l-y_l)\omega_l}d\omega_l = \delta(x_l-y_l)$, $1\leq l \leq k$, and $\frac{1}{2\pi}\int_{\R}e^{g (x_m-y_m)\omega_m}d\omega_m = \delta(x_m-y_m)$, $k+1\leq m \leq n$, for the seventh equality, and again \eqref{eq:symsc} for the last equality. Division of both sides with $(2\pi)^n$ finally gives the Plancherel identity \eqref{eq:Planch}.
\end{proof}

\subsection{Convolution}

The properties of convolutions subject to several types of Fourier transforms in Clifford algebra have been recently studied in \cite{BSH:GFTConv,BBSS:ConvHypCpxFT}. 
We define the \textit{convolution} of two multivector signals $a,b \in L^1(R^{p,q}; \cl{p,q})$ as
\be 
  (a\star b)(\bvect{x}) = \int_{\R^{p,q}}a(\bvect{y})b(\bvect{x}-\bvect{y})d^n\bvect{y},
  \label{eq:conv}
\end{equation} 
provided that the integral exists.
For establishing the general two-sided CFT of the convolution, we need the identity
\be 
  e^{\alpha f}e^{\beta g}
  = e^{\beta g}e^{\alpha f} 
    + [f,g]\sin(\alpha)\sin(\beta), 
  \qquad 
  [f,g]=fg-gf. 
  \label{eq:expexpid}
\end{equation} 
We further define the following two mixed exponential-sine transforms
\begin{align} 
  \mathcal{F}^{f,\pm s}\{ h \}(\bomega)
  &= \int_{\R^{p,q}} e^{-f u(\svect{x},\bomega)}h(\bvect{x}) (\pm 1)\sin(-v(\bvect{x},\bomega)) d^n\bvect{x},
  \label{eq:CFTexpsin}
  \\
  \mathcal{F}^{\pm s,g}\{ h \}(\bomega)
  &= \int_{\R^{p,q}} (\pm 1)\sin(-u(\bvect{x},\bomega)) h(\bvect{x}) e^{-g v(\svect{x},\bomega)} d^n\bvect{x}.
  \label{eq:CFTsinexp}
\end{align} 
We assume that the functions $u,v$ are both linear with respect to their first argument, but we do not assume the special forms of $u,v$ stated in \eqref{eq:assumeuv}.
\begin{thm}[CFT of convolution]
Let the functions $u,v$ be both linear with respect to their first argument. 
The \textit{general two-sided CFT of the convolution} (\ref{eq:conv}) of two functions $a,b \in L^1(R^{p,q}; \cl{p,q})$ can then be expressed as
\begin{align}
  &\mathcal{F}^{f,g}\{ a\star b \}(\bomega) = \nonumber \\
  &\phantom{+ } \mathcal{F}^{f,g}\{ a_{+f} \}(\bomega)  \mathcal{F}^{f,g}\{ b_{+g} \}(\bomega) 
  +  \mathcal{F}^{f,-g}\{ a_{+f} \}(\bomega) \mathcal{F}^{f,g}\{ b_{-g} \}(\bomega) 
  \nonumber \\
  &+ \mathcal{F}^{f,g}\{ a_{-f} \}(\bomega)  \mathcal{F}^{-f,g}\{ b_{+g} \}(\bomega)
  +  \mathcal{F}^{f,-g}\{ a_{-f} \}(\bomega) \mathcal{F}^{-f,g}\{ b_{-g} \}(\bomega)
  \label{eq:CFTconv}\\
  &+ \mathcal{F}^{f,s}\{ a_{+f} \}(\bomega)  [f,g] \mathcal{F}^{s,g}\{ b_{+g} \}(\bomega)
  +  \mathcal{F}^{f,-s}\{ a_{+f} \}(\bomega) [f,g] \mathcal{F}^{s,g}\{ b_{-g} \}(\bomega)
  \nonumber \\
  &+ \mathcal{F}^{f,s}\{ a_{-f} \}(\bomega)  [f,g] \mathcal{F}^{-s,g}\{ b_{+g} \}(\bomega)
  +  \mathcal{F}^{f,-s}\{ a_{-f} \}(\bomega) [f,g] \mathcal{F}^{-s,g}\{ b_{-g} \}(\bomega).
  \nonumber 
\end{align}
\end{thm}
\begin{proof}
\noindent
We now prove \eqref{eq:CFTconv}. 
\begin{align}
  &\mathcal{F}^{f,g}\{ a\star b \}(\bomega) 
  \nonumber \\
  &= \int_{\R^{p,q}} e^{-f u(\svect{x},\bomega)} (a\star b)(\bvect{x}) 
              \,e^{-g v(\svect{x},\bomega)} d^n\bvect{x}
  \nonumber \\
  &= \int_{\R^{p,q}} \int_{\R^{p,q}}
    e^{-f u(\svect{x},\bomega)} 
         a(\bvect{y})b(\bvect{x}-\bvect{y}) d^n\bvect{y} 
              \,e^{-g v(\svect{x},\bomega)} d^n\bvect{x}
  \nonumber \\
  &= \int_{\R^{p,q}} \int_{\R^{p,q}}
    e^{-f u(\svect{y}+\svect{z},\bomega)} 
         a(\bvect{y})  b(\bvect{z})           d^n\bvect{y} 
              \,e^{-g v(\svect{y}+\svect{z},\bomega)} d^n\bvect{z} 
  \nonumber \\
  &= \int_{\R^{p,q}} \int_{\R^{p,q}}
    e^{-f u(\svect{y},\bomega)} 
    e^{-f u(\svect{z},\bomega)}
         a(\bvect{y})  b(\bvect{z})           d^n\bvect{y} 
              \,e^{-g v(\svect{y},\bomega)}
                e^{-g v(\svect{z},\bomega)} d^n\bvect{z}
  \nonumber \\
  &= \int_{\R^{p,q}} \int_{\R^{p,q}}
    e^{-f u(\svect{y},\bomega)} 
    e^{-f u(\svect{z},\bomega)}
         [a_{+f}(\bvect{y})+a_{-f}(\bvect{y}) ] 
  \nonumber \\
  &\phantom{=\int_{\R^{p,q}} \int_{\R^{p,q}}} 
         [b_{+g}(\bvect{z})+b_{-g}(\bvect{z}) ] d^n\bvect{y} 
              \,e^{-g v(\svect{y},\bomega)}
                e^{-g v(\svect{z},\bomega)} d^n\bvect{z},
  \label{eq:proofconv1}
\end{align}
where we used the substitution
$\bvect{z}=\bvect{x}-\bvect{y}$, $\bvect{x}=\bvect{y}+\bvect{z}$. To simplify \eqref{eq:proofconv1} we expand the inner expression of the integrand and insert \eqref{eq:expexpid} to obtain
\begin{align}
  &e^{-f u(\svect{z},\bomega)}
         [a_{+f}(\bvect{y})+a_{-f}(\bvect{y}) ]  
         [b_{+g}(\bvect{z})+b_{-g}(\bvect{z}) ]          
              \,e^{-g v(\svect{y},\bomega)}
  \nonumber \\
  &= [a_{+f}(\bvect{y})e^{-f u(\svect{z},\bomega)}
     +a_{-f}(\bvect{y})e^{+f u(\svect{z},\bomega)} ]  
    [e^{-g v(\svect{y},\bomega)} b_{+g}(\bvect{z})
     + e^{+g v(\svect{y},\bomega)} b_{-g}(\bvect{z}) ]  
  \nonumber \\
  &= a_{+f}(\bvect{y}) e^{-f u(\svect{z},\bomega)}e^{-g v(\svect{y},\bomega)}
    b_{+g}(\bvect{z})
    + 
    a_{-f}(\bvect{y}) e^{+f u(\svect{z},\bomega)}e^{-g v(\svect{y},\bomega)}
    b_{+g}(\bvect{z}) 
    \nonumber \\
  &\phantom{=}+
    a_{+f}(\bvect{y}) e^{-f u(\svect{z},\bomega)}e^{+g v(\svect{y},\bomega)}
    b_{-g}(\bvect{z})   
    +
    a_{-f}(\bvect{y}) e^{+g v(\svect{y},\bomega)}e^{+f u(\svect{z},\bomega)}
    b_{-g}(\bvect{z})    
  \nonumber \\
  &= a_{+f}(\bvect{y}) \{e^{-g v(\svect{y},\bomega)}e^{-f u(\svect{z},\bomega)}
    +[f,g]\sin(-v)\sin(-u)\}
    b_{+g}(\bvect{z})
    \nonumber \\
  &\phantom{=}+ 
    a_{-f}(\bvect{y}) \{e^{-g v(\svect{y},\bomega)}e^{+f u(\svect{z},\bomega)}
    +[f,g]\sin(-v) (-1)\sin(-u)\}
    b_{+g}(\bvect{z}) 
    \nonumber \\
  &\phantom{=}+
    a_{+f}(\bvect{y}) \{e^{+g v(\svect{y},\bomega)}e^{-f u(\svect{z},\bomega)}
    +[f,g](-1)\sin(-v) \sin(-u)\}
    b_{-g}(\bvect{z})   
    \nonumber \\
  &\phantom{=}+
    a_{-f}(\bvect{y}) \{e^{+g v(\svect{y},\bomega)}e^{+f u(\svect{z},\bomega)}
    +[f,g](-1)\sin(-v) (-1)\sin(-u)\}
    b_{-g}(\bvect{z})  
  \nonumber \\
  &= a_{+f}(\bvect{y})e^{-g v(\svect{y},\bomega)}\,
    e^{-f u(\svect{z},\bomega)}b_{+g}(\bvect{z})
    + 
    a_{-f}(\bvect{y})e^{-g v(\svect{y},\bomega)}\,
    e^{+f u(\svect{z},\bomega)}b_{+g}(\bvect{z}) 
    \nonumber \\
  &\phantom{=}+
    a_{+f}(\bvect{y})e^{+g v(\svect{y},\bomega)}\,
    e^{-f u(\svect{z},\bomega)}b_{-g}(\bvect{z})   
    +
    a_{-f}(\bvect{y})e^{+g v(\svect{y},\bomega)}\,
    e^{+f u(\svect{z},\bomega)}b_{-g}(\bvect{z}) 
    \nonumber \\
  &\phantom{=}+
    a_{+f}(\bvect{y})\sin(-v)\, [f,g]\, \sin(-u) b_{+g}(\bvect{z})
    \nonumber \\
  &\phantom{=}+ 
    a_{-f}(\bvect{y})\sin(-v)\, [f,g]\, (-1)\sin(-u)b_{+g}(\bvect{z}) 
    \nonumber \\
  &\phantom{=}+
    a_{+f}(\bvect{y})(-1)\sin(-v)\, [f,g]\, \sin(-u) b_{-g}(\bvect{z})
    \nonumber \\
  &\phantom{=}+
    a_{-f}(\bvect{y})(-1)\sin(-v)\, [f,g]\, (-1)\sin(-u) b_{-g}(\bvect{z}).
  \label{eq:proofconv2}
\end{align}
Reinserting \eqref{eq:proofconv2} into \eqref{eq:proofconv1} and subsequently inserting the definitions \eqref{eq:CFTexpsin} and \eqref{eq:CFTsinexp} we get
\begin{align}
  &\mathcal{F}^{f,g}\{ a\star b \}(\bomega)
  \nonumber \\
  &= \int_{\R^{p,q}} 
    e^{-f u(\svect{y},\bomega)} 
    a_{+f}(\bvect{y})e^{-g v(\svect{y},\bomega)}\,d^n\bvect{y} 
    \int_{\R^{p,q}}
    e^{-f u(\svect{z},\bomega)}b_{+g}(\bvect{z})
    e^{-g v(\svect{z},\bomega)} d^n\bvect{z}
    \nonumber \\
  &\phantom{=}+ \int_{\R^{p,q}} 
    e^{-f u(\svect{y},\bomega)} 
    a_{-f}(\bvect{y})e^{-g v(\svect{y},\bomega)}\,d^n\bvect{y} 
    \int_{\R^{p,q}}
    e^{+f u(\svect{z},\bomega)}b_{+g}(\bvect{z}) 
    e^{-g v(\svect{z},\bomega)} d^n\bvect{z}
    \nonumber \\
  &\phantom{=}+ \int_{\R^{p,q}} 
    e^{-f u(\svect{y},\bomega)} 
    a_{+f}(\bvect{y})e^{+g v(\svect{y},\bomega)}\,d^n\bvect{y} 
    \int_{\R^{p,q}}
    e^{-f u(\svect{z},\bomega)}b_{-g}(\bvect{z})   
    e^{-g v(\svect{z},\bomega)} d^n\bvect{z}
    \nonumber \\
  &\phantom{=}+ \int_{\R^{p,q}} 
    e^{-f u(\svect{y},\bomega)} 
    a_{-f}(\bvect{y})e^{+g v(\svect{y},\bomega)}\,d^n\bvect{y} 
    \int_{\R^{p,q}}
    e^{+f u(\svect{z},\bomega)}b_{-g}(\bvect{z}) 
    e^{-g v(\svect{z},\bomega)} d^n\bvect{z}
    \nonumber \\
  &\phantom{=}+ \int_{\R^{p,q}} 
    e^{-f u(\svect{y},\bomega)} 
    a_{+f}(\bvect{y})\sin(-v)\,d^n\bvect{y}  \,[f,g]\, 
    \int_{\R^{p,q}} \sin(-u) b_{+g}(\bvect{z})
    e^{-g v(\svect{z},\bomega)} d^n\bvect{z}
    \nonumber \\
  &\phantom{=}+ \int_{\R^{p,q}} 
    e^{-f u(\svect{y},\bomega)} 
    a_{-f}(\bvect{y})\sin(-v)\,d^n\bvect{y}  \,[f,g]\, 
    \int_{\R^{p,q}} (-1)\sin(-u)b_{+g}(\bvect{z}) 
    e^{-g v(\svect{z},\bomega)} d^n\bvect{z}
    \nonumber \\
  &\phantom{=}+ \int_{\R^{p,q}} 
    e^{-f u(\svect{y},\bomega)} 
    a_{+f}(\bvect{y})(-1)\sin(-v)\,d^n\bvect{y}  \,[f,g]\, 
    \int_{\R^{p,q}} \sin(-u) b_{-g}(\bvect{z})
    e^{-g v(\svect{z},\bomega)} d^n\bvect{z}
    \nonumber \\
  &\phantom{=}+ \int_{\R^{p,q}} 
    e^{-f u(\svect{y},\bomega)} 
    a_{-f}(\bvect{y})(-1)\sin(-v)\,d^n\bvect{y}  \,[f,g]\, 
    \nonumber \\
  &\phantom{= + \int_{\R^{p,q}} }
    \int_{\R^{p,q}} (-1)\sin(-u) b_{-g}(\bvect{z})
    e^{-g v(\svect{z},\bomega)} d^n\bvect{z}
  \nonumber \\
  &= 
  \mathcal{F}^{f,g}\{ a_{+f} \}(\bomega)  \mathcal{F}^{f,g}\{ b_{+g} \}(\bomega) 
  +  \mathcal{F}^{f,-g}\{ a_{+f} \}(\bomega) \mathcal{F}^{f,g}\{ b_{-g} \}(\bomega) 
  \nonumber \\
  &\phantom{=}+ \mathcal{F}^{f,g}\{ a_{-f} \}(\bomega)  \mathcal{F}^{-f,g}\{ b_{+g} \}(\bomega)
  +  \mathcal{F}^{f,-g}\{ a_{-f} \}(\bomega) \mathcal{F}^{-f,g}\{ b_{-g} \}(\bomega)
  \\
  &\phantom{=}+ \mathcal{F}^{f,s}\{ a_{+f} \}(\bomega)  [f,g] \mathcal{F}^{s,g}\{ b_{+g} \}(\bomega)
  +  \mathcal{F}^{f,-s}\{ a_{+f} \}(\bomega) [f,g] \mathcal{F}^{s,g}\{ b_{-g} \}(\bomega)
  \nonumber \\
  &\phantom{=}+ \mathcal{F}^{f,s}\{ a_{-f} \}(\bomega)  [f,g] \mathcal{F}^{-s,g}\{ b_{+g} \}(\bomega)
  +  \mathcal{F}^{f,-s}\{ a_{-f} \}(\bomega) [f,g] \mathcal{F}^{-s,g}\{ b_{-g} \}(\bomega).
  \nonumber 
  \label{eq:proofconv3}
\end{align}
\end{proof}

\section{Conclusions}

We have established a comprehensive \textit{new mathematical framework} for the investigation and application of Clifford Fourier transforms (CFTs) together with \textit{new properties}. Our new CFTs form a more general class of CFTs, subsuming and generalizing previous results. We have applied new results on square roots of $-1$ in Clifford algebras to fully general construct CFTs, with two general square roots of $-1$ in Clifford algebras $\cl{p,q}$. The new CFTs are \textit{fully steerable} within the continuous Clifford algebra submanifolds of square roots of $-1$. We have thus  left the terra cognita of familiar transforms to outline the vast array of possible CFTs in $\cl{p,q}$.  

We first reviewed the recent results on \textit{square roots of $-1$} in Clifford algebras. We then showed how the $\pm$ split or orthogonal 2D planes split of quaternions can be generalized to \textit{split multivector signal functions} with respect to a general pair of square roots of $-1$ in Clifford algebra. Next, we defined the central notion of \textit{general two-sided Clifford Fourier transforms} with respect to any two square roots of $-1$ in Clifford algebra. Finally, we investigated important \textit{properties} of these new CFTs: linearity, shift, modulation, dilation, moments, inversion, derivatives, Plancherel and Parseval formulas, as well as a convolution theorem.

Regarding numerical implementations, Theorem \ref{th:fpmtrafo} shows that $2^n$ complex Fourier transformations (FTs) are sufficient. In some cases this can be reduced to $2^{(n-1)}$ complex FTs, e.g., when one of the two square roots of $-1$ is a pseudoscalar. Further algebraic studies may widen the class of CFTs, where $2^{(n-1)}$ complex FTs are sufficient. Numerical implementation is then possible with $2^n$ (or $2^{(n-1)}$) discrete complex FTs, which can also be fast Fourier transforms (FFTs), leading to fast CFT implementations. 

A well-known example of a CFT with two square roots of $-1$ are the quaternion FTs (QFTs) \cite{TAE:QFT2dLinPDF,EH:QFTgen,EH:OPS-QFT,TB:PhD,BFS:HypcFT,SJS:FTcolQuat,ES:HypCFTColIm}, which are particularly used in applications to partial differential systems, color image processing,
filtering, disparity estimation (two images differ by local translations), and texture segmentation. Another example is the spacetime FT, which leads to a multivector wave packet analysis of spacetime signals (e.g. electro-magnetic signals), applicable even to relativistic signals \cite{EH:QFTgen,EH:DirQFTUP}.

Depending on the choice of the phase functions $u(\bvect{x},\bomega)$ and $v(\bvect{x},\bomega)$, the multivector basis coefficient functions of the CFT result carry information on the symmetry of the signal, similar to the special case of the QFT \cite{TB:PhD}.

The convolution theorem allows to design and apply multivector valued filters to multivector valued signals. 

\section*{Acknowledgment}

E. H. thanks God: \textit{Soli deo gloria!}, his family, J. Helmstetter, R. Ab{\l}amowicz, S. Sangwine, the anonymous reviewers and the AGACSE 2012 organizers. He further thanks R. Bujack for very helpful comments.

\end{document}